\documentclass{article}
\usepackage{graphicx} 
\usepackage{amsthm}
\usepackage[hidelinks]{hyperref}
\usepackage{amsmath}
\usepackage{amssymb}
\usepackage{setspace}
\usepackage{cite}
\usepackage{ragged2e}
\newtheorem{theorem}{Theorem}[section]
\newtheorem{observation}{Observation}[section]
\newtheorem{definition}{Definition}[section]
\newtheorem{corollary}{Corollary}[section]
\title{The Dual-Server Domination Number of Some Graph Operations}
\author{
Shrilaxmi Laxminarayana Rao$^{1}$,
Sayinath Udupa N. V.$^{1}$\footnote{Corresponding author}, Prathviraj N.$^{2}$\\[1ex]
$^{1}$Manipal Institute of Technology,\\
Manipal Academy of Higher Education, Manipal, India\\
\texttt{shrilaxmi.mitmpl2025@learner.manipal.edu}\\
\texttt{sayinath.udupa@manipal.edu}\\[1ex]
$^{2}$Manipal School of Information Sciences,\\
Manipal Academy of Higher Education, Manipal, India\\
\texttt{prathviraj.n@manipal.edu}
}
\date{}
\begin{document}
\justifying
\maketitle
\begin{abstract}
A subset $S \subseteq V(G)$ is called a dual-server dominating set if there exists a partition $\pi_S=\{R,B\}$ of $S$ such that every vertex in $V(G)\setminus S$ has at least one neighbour in $R$ and at least one neighbour in $B.$ The minimum cardinality of a DS-dominating set of $G$ is called the dual-server domination number, denoted by $\gamma_{ds}(G)$. In this paper, we investigate the dual-server domination number of several graph operations. We establish exact values for the dual-server domination number of the join, Cartesian product, and corona of graphs. We also determine the dual-server domination number of the splitting graphs of paths, cycles, and complete bipartite graphs. These results extend the study of dual-server domination and provide further insight into its behaviour under graph operations.
\end{abstract}
\textbf{Keywords:} Domination, Dual-server domination, Splitting graph, Graph operations.\\
\noindent\textbf{2020 Mathematics Subject Classification:}
05C69, 05C75, 05C76, 05C90.
\section{Introduction} Domination and its numerous variants have been widely studied in graph theory. For a comprehensive treatment of domination theory and its fundamental concepts, one may refer to\cite{Haynes1998,Haynes2023,HBW,Slater1998}. Throughout this paper, all graphs are finite, simple and undirected. Let $G=(V,E)$ be a graph with vertex set $V(G)$ and edge set $E(G).$ Two vertices $u$ and $v$ in $G$ are adjacent if $e = uv \in E(G).$ The open neighbourhood of a vertex $v$ in $G$ is denoted by $N_G(v)$ and is defined as
$N_G(v)=\{u \in V(G) \mid uv \in E(G)\}$. The minimum and maximum degrees of $G$ are denoted by $\delta(G)$ and $\Delta(G)$, respectively. A set $ S\subseteq V(G)$ is a dominating set if every vertex in $V(G) \setminus S$ has a neighbour in $S.$ The minimum cardinality of a dominating set of $G$ is called the domination number of $G$ and is denoted by $\gamma(G)$. The path graph, cycle graph, complete graph, edgeless graph, and complete bipartite graph on $n$ vertices are denoted by $P_n$, $C_n$, ${K_n}$, $\overline{K_n}$ and $K_{m,n}$, respectively. Several variants of domination have been introduced and extensively studied in the literature to model different graph-theoretic and network-related problems. A variety of domination parameters have been investigated, including connected domination, total domination, double domination, Roman domination and many other variants\cite{Haynes2023}. Recently Chellali et al.\cite{DS,TDS} introduced dual-server domination  as a model for networks that provide two distinct types of services. A set $S \subseteq V(G)$ is called a dual-server dominating set (or DS-dominating set) if there exists a partition $\pi_S=\{R,B\}$ of $S$ such that every vertex in $V(G)\setminus S$ has at least one neighbour in $R$ and at least one neighbour in $B.$ The minimum cardinality of a DS-dominating set of $G$ is called the dual-server domination number, denoted by $\gamma_{ds}(G).$ Throughout this paper, we write $S=R \cup B$ to denote a DS-dominating set together with its associated partition $\{R,B\}$. The study of graph parameters under graph operations has received considerable attention. In particular, domination-related parameters have been investigated for several graph operations, including the join, Cartesian product and corona\cite{West}. Among graph transformations, the splitting graph introduced by Sampathkumar and Walikar\cite{SPG} has been studied extensively with respect to several graph parameters, including domination-related parameters\cite{Deepalaxmi}. In this paper, we determine the dual-server domination number of the splitting graphs of paths, cycles, complete graphs, and complete bipartite graphs. We further study the dual-server domination number under the join, corona, and Cartesian product operations. The paper is organized as follows. Section 2 is devoted to the dual-server domination number of splitting graphs, while Section 3 considers the dual-server domination number under the join, corona, and Cartesian product operations.

\section{Dual-Server Domination in Splitting Graphs}

In this section, we study the dual-server domination number of splitting 
graphs. The splitting graph $Sp(G)$ is obtained by adding, for each vertex 
$v\in V(G)$, a new vertex $v'$, called the twin vertex of $v$, such that $N(v')=N(v).$

We determine the dual-server domination number of splitting graphs of paths, 
cycles, complete graphs, and complete bipartite graphs.

\begin{theorem}\label{SPN}
For the path graph $P_n$, $\gamma_{ds}(S_p(P_n)) = n+1.$
\end{theorem}

\begin{proof}
Let $P_n$ be the path graph on $n$ vertices with vertex set $V(P_n)= \{v_1, v_2, ..., v_n\}$. Let $V'= \{v_1', v_2', ..., v_n'\}$ denote the twin vertices corresponding to $v_1, v_2, ..., v_n$ in the splitting graph $S_p(P_n)$. The neighbourhoods in $S_p(P_n)$ are $N(v_1')= \{v_2\}$, $N(v_n')= \{v_{n-1}\}$, $N(v_i')= \{v_{i-1}, v_{i+1}\}$, $2 \le i \le n-1$, and $N(v_i)= \{v_{i-1}, v_{i+1}, v'_{i-1}, v'_{i+1}\}$, $2 \le i \le n-1$. Since both $v_1'$ and $v_n'$ have a unique neighbour, each of these vertices must belong to every dual-server dominating set. Let $S = R \cup B$. When $n$ is odd, define $S=\{v_i, v_i': i \text{ is odd}\}$. When $n$ is even, define $S= \{v_i, v_i': i \text{ is odd}\} \cup \{v_n\}$, where $R= \{v_i: i \equiv1\pmod{4}\} \cup \{v_i': i \equiv3\pmod{4}\}$, and $B= \{v_i': i \equiv1\pmod{4}\} \cup \{v_i: i \equiv3\pmod{4}\}$.  When $n$ is even, assign $v_n$ to either $R$ or $B$. Clearly, $|S| = n + 1$.\\ We now show that $S$ is a dual-server dominating set.  Let $v_i \notin S$. Since both $i-1$ and $i+1$ are odd, the construction of $R$ and $B$ ensures that at least one of the neighbours of $v_i$ belongs to $R$ and $B$. \\ If $v_i' \notin S$, then  its two neighbours are ${v_{i-1}, v_{i+1}}$. Since $i-1$ and $i+1$ are odd, by the construction of $R$ and $B$, at least one of the neighbours of $v_i$ belongs to $R$ and $B$. Thus, every vertex outside $S$ has at least one neighbour in both $R$ and $B$. Therefore, $S$ is a dual-server dominating set and hence $\gamma_{ds}(S_p(P_n)) \le n+1$. \\ Next, we show that $\gamma_{ds}(S_p(P_n)) \geq n+1$.\\ Let $S = R \cup B$ be a minimum dual-server dominating set of $S_p(P_n)$. Since $v_1'$ and $v_n'$ are pendant vertices, they must belong to S. Consider a vertex $v_i' \notin S$, where $i$ is even. Since $N(v_i')= \{v_{i-1}, v_{i+1}\}$, both of its neighbours must belong to S, with one neighbour in $R$ and the other in $B$. Consequently, all odd-indexed original vertices must belong to $S$. There are $\left\lceil \frac{n}{2} \right\rceil$ such vertices. Further, to ensure the domination of the remaining vertices, all odd-indexed twin vertices must also belong to  $S$. There are $\left\lceil \frac{n}{2} \right\rceil$ such vertices. Therefore, $ |S| \ge \left\lceil \frac{n}{2} \right\rceil + \left\lceil \frac{n}{2} \right\rceil = 2 \left\lceil \frac{n}{2} \right\rceil $.\\ If $n= 2k+1$, then \[|S| \ge 2 \left\lceil \frac{2k+1}{2}\right\rceil. \]
Since $k$ is an integer, 
$ 2 \left\lceil \frac{2k+1}{2} \right\rceil = 2(k+1) = n + 1$.
Hence $ \gamma_{ds}(S_p(p_n))= |S| \ge n + 1 $.\\ If $n= 2k$, then \[ |S| \ge 2 \left\lceil \frac{2k}{2} \right\rceil = 2k =n.\] Since $v_n'$ is a pendant vertex, we obtain $ \gamma_{ds}(S_p(p_n))= |S| \ge n + 1 $. Together with the upper bound, it follows that\\ \[\gamma_{ds}(S_p(P_n)) = n+1.\]
\end{proof}

\begin{corollary}
For the path graph $P_n$, $\gamma_{ds}(P_n)= \lceil \frac{\gamma_{ds}(S_p(P_n))}{2} \rceil.$
\end{corollary}
\begin{proof}
From theorem~\ref{SPN}, $\gamma_{ds}(S_p(P_n)) = n+1.$ Also, $\gamma_{ds}(P_n)= \lceil \frac{n+1}{2} \rceil$ (Corollary 7 of \cite{DS}). Therefore, $\lceil \frac{\gamma_{ds}(S_p(P_n))}{2} \rceil= \lceil \frac{n+1}{2} \rceil= \gamma_{ds}(P_n)$. 
\end{proof}
\begin{theorem}\label{SCN}
For the cycle graph $C_n$,
\[
\gamma_{ds}(S_p(C_n))=
\begin{cases}
n, & \text{if } n\equiv 0 \pmod{4},\\
n+1, & \text{otherwise}.
\end{cases}
\]

\end{theorem}
\begin{proof}
Let $C_n$ be the cycle graph on $n$ vertices with vertex set $V(C_n)= \{v_1, v_2, ..., v_n\}$. Let $V^\prime= \{v_1^\prime, v_2^\prime, ..., v_n^\prime\}$ denote the twin vertices corresponding to $v_1, v_2, ..., v_n$ in the splitting graph $S_p(C_n)$. Now we shall prove the result using following cases.\\

\textbf{Case 1}. $n= 4k$.\\ Let $S= R \cup B$, where $R= \{v_{4i+1}, v_{4i+3}', 0 \le i \le k-1$\} and $B= \{v_{4i+3}, v_{4i+1}', 0 \le i \le k-1$\}. Clearly $|R|= \frac{n}{2}$, $|B|= \frac{n}{2}$ and hence $|S|= n$.\\ Now let $v_i' \notin S$. Since $N(v_i')= \{v_{i-1}, v_{i+1}\}$, it follows from the above construction that $v_i'$ has at least one neighbour in $R$ and $B$. Similarly, let $v_i \notin S$. Since $N(v_i)= \{v_{i-1}, v_{i+1}, v'_{i-1}, v'_{i+1}\}$, the construction ensures that $v_i$ has at least one neighbour in $R$ and $B$. Therefore, every vertex outside $S$ has a neighbour in both $R$ and $B$. Hence, S is DS-dominating set of $S_p(C_n)$. Consequently, $\gamma_{ds}(S_p(C_n)) \leq n$.\\

Next, we prove that $\gamma_{ds}(S_p(C_n)) \geq n$.\\ Let $S= R \cup B$  be a minimum DS-dominating set of $S_p(C_n)$. Consider an even-indexed twin vertex $v_i' \notin S$. Since $N(v_i')= \{v_{i-1}, v_{i+1}\}$, the definition of a dual-server dominating set implies that both $v_{i-1}$ and  $v_{i+1}$ must belong to $S$, with one vertex in $R$ and the other in $B$. As $i$ ranges over all even indices, every odd-indexed original vertex is therefore forced to belong to $S$. Hence, $S$ contains $ \frac{n}{2} $ odd-indexed original vertices. Further, to ensure domination of the remaining vertices, all odd-indexed twin vertices must also belong to $S$. Hence, $S$ contains $\frac{n}{2} $ odd-indexed twin vertices. Therefore,
\[\gamma_{ds}(S_p(C_n))=|S| \geq  \frac{n}{2} + \frac{n}{2} =n. \] 

\textbf{Case 2}. $n= 4k+1$ or $n= 4k+3$.\\ Consider the same construction of $S$ as in case 1, where $R= \{v_{4i+1}, v_{4i+3}', 0 \le i \le k-1$\} and $B= \{v_{4i+3}, v_{4i+1}', 0 \le i \le k-1$\}. Since $n$ is odd, there are $\left\lceil \frac{n}{2} \right\rceil$ odd-indexed original and twin vertices. Hence \\\[ |S|= 2\left\lceil \frac{n}{2} \right\rceil= n+1. \]\\The verification that $S$ is a DS-dominating set of $S_p(C_n)$ is identical to that of Case 1. Hence, $\gamma_{ds}(S_p(C_n)) \leq n+1$. \\ Next, we show that $\gamma_{ds}(S_p(C_n)) \geq n+1$. 
\\Let $S= R \cup B$ be a minimum dual-server dominating set of $S_p(C_n)$. By the same argument as in case 1, every odd-indexed original and twin vertex are forced to belong to $S$. Hence, $|S| \geq \left\lceil \frac{n}{2} \right\rceil+\left\lceil \frac{n}{2} \right\rceil= 2\left\lceil \frac{n}{2} \right\rceil$, Since n is odd, we obtain $|S| \geq n+1$. Therefore $\gamma_{ds}(S_p(C_n)) \geq n+1$. 

\textbf{Case 3}. n= 4k+2.\\
Let $S= \{v_i, v_i': i \text{~is odd}\} \cup \{v_n'\}$, where $R= \{v_{4i+1}, v_{4i+3}', 0 \le i \le k-1$\} and $B= \{v_{4i+3}, v_{4i+1}', 0 \le i \le k-1$\}. without loss of generality, the vertex $v_n'$ may be placed either in $R$ or $B$. Clearly, $|S|= n+1$. By the same argument as in case 1, every vertex outside $S$ has atleast one neighbour in $R$ and $B$. Therefore, $S$ is dual-server dominating set of $S_p(C_n)$. Hence, $\gamma_{ds}(S_p(C_n)) \leq n+1$.\\

Next, we show that $\gamma_{ds}(S_p(C_n)) \geq n+1$ .\\ Let $S= R \cup B $ be a minimum DS-dominating set of $S_p(C_n)$. By the same argument as in case 1, every odd- indexed original and twin vertices are forced to belong to S. Hence, $|S| \geq  \frac{n}{2} + \frac{n}{2} =n$.\\Suppose to the contrary, that $|S|= n$. Then $S$ consists precisely of all odd-indexed vertices. The dual-server domination condition forces the coloring of successive odd-indexed pairs throughout the cycle. Since $n= 4k+2$, there are $2k+1$ odd-indexed pairs. Therefore, $v_1$ and $v_{n-1}$ belong to the same color class. Now, consider the twin vertex $v_n' \notin S$. Since $N(v_n')= \{v_{n-1}, v_1\}$, both its neighbours belong to the same color class. Thus, $v_n'$ does not have neighbour in one color class. Hence, no DS-dominating set of cardinality $n$ exists. Therefore, $|S| \geq n+1$ and consequently, $\gamma_{ds}(S_p(C_n)) \geq n+1$.\\ Combining this with the upper bound, we obtain  \[ \gamma_{ds}(S_p(C_n))=  n+1.\]
\end{proof}

\begin{corollary}
For the cycle graph $C_n$, $\gamma_{ds}(C_n)= \lceil \frac{\gamma_{ds}(S_p(C_n))}{2} \rceil.$
\end{corollary}
\begin{proof}
From theorem~\ref{SCN}, $\gamma_{ds}(S_p(C_n))=
\begin{cases}
n, & \text{if } n\equiv 0 \pmod{4},\\
n+1, & \text{otherwise}.
\end{cases}$. Also, $\gamma_{ds}(C_n)=
\begin{cases}
\frac{n+2}{2}, & \text{if } n\equiv 2 \pmod{4},\\
\lceil \frac{n}{2} \rceil, & \text{otherwise}.
\end{cases}$ (Proposition 8 of \cite{DS}).\\We consider two cases.

\textbf{Case 1.} Suppose $n\equiv 0 \pmod{4}$. Then $\gamma_{ds}(S_p(C_n))= n.$ Therefore, $\lceil \frac{\gamma_{ds}(S_p(C_n))}{2} \rceil= \lceil \frac{n}{2} \rceil= \frac{n}{2}= \gamma_{ds}(C_n)$.  \\

\textbf{Case 2.} Suppose $n\not\equiv 0 \pmod{4}$. Then $\gamma_{ds}(S_p(C_n))= n+1.$\\ If $n\equiv 2 \pmod{4}$, let $n= 4k+2$ for some integer $k>0$. Then $\lceil \frac{\gamma_{ds}(S_p(C_n))}{2} \rceil= \lceil \frac{n+1}{2} \rceil= \lceil \frac{4k+3}{2} \rceil= 2k+2= \frac{n+2}{2}= \gamma_{ds}(C_n)$.\\ If $n\equiv 1 \text{ or } 3 \pmod{4}$, then $\lceil \frac{\gamma_{ds}(S_p(C_n))}{2} \rceil= \lceil \frac{n+1}{2} \rceil= \frac{n}{2}= \gamma_{ds}(C_n)$.

Hence, $\gamma_{ds}(C_n)= \lceil \frac{\gamma_{ds}(S_p(C_n))}{2} \rceil.$

\end{proof}
\begin{theorem}
For $m, n\geq 2$, $\gamma_{ds}(S_p(K_{m,n}))= 4.$
\end{theorem}
\begin{proof}
Let $K_{m, n}$ be a complete bipartite graph with bipartition sets $X= \{v_1, v_2,..., v_m\}$ and $Y= \{u_1, u_2,..., u_n\}$. Let $X'= \{v_1', v_2',..., v_m'\}$ and $Y'= \{u_1', u_2',..., u_n'\}$ denote the sets of twin vertices corresponding to $X$ and $Y$, respectively, in the splitting graph $S_p(K_{m, n})$. We note that
\[
N(v_i)=Y\cup Y', \qquad
N(u_j)=X\cup X', \qquad
N(v_i')=Y, \qquad
N(u_j')=X.
\]
Let
$S=R\cup B=\{u_1,u_2,v_1,v_2\},
$
where
$
R=\{u_1,v_1\}$ and $ 
B=\{u_2,v_2\}.
$
Clearly, $|S|=4$.

Now we consider the following cases.

\textbf{Case 1.}
 $v_i\notin S$.\\ Since
$
N(v_i)=Y\cup Y'$, and  $u_1\in R$ and $u_2\in B$, it follows that $v_i$ has at least one
neighbour in both $R$ and $B$.

\textbf{Case 2.}
$v_i'\notin S$.\\ Since
$N(v_i')=Y$, and $u_1\in R$ and $u_2\in B$, each vertex $v_i'$ has at least one
neighbour in both $R$ and $B$.

\textbf{Case 3.}  $u_j\notin S$.\\ Since
$N(u_j)=X\cup X'$, and  $v_1\in R$ and $v_2\in B$, every such vertex $u_j$ has at least one
neighbour in both $R$ and $B$.

\textbf{Case 4.} $u_j'\notin S$.\\ Since $
N(u_j')=X$, and  $v_1\in R$ and $v_2\in B$, each vertex $u_j'$ has one neighbour in
both $R$ and $B$.

Hence, every vertex outside $S$ has at least one neighbour in $R$ and $B$. Therefore, $S$ is a dual-server dominating set  of
$S_p(K_{m,n})$.

Thus,
\[
\gamma_{ds}(S_p(K_{m,n}))\le 4.
\]
Next, we show that $
\gamma_{ds}(S_p(K_{m,n}))\ge 4$.\\
Assume, for contradiction, that $
\gamma_{ds}(S_p(K_{m,n}))\le 3$.
Let $
S=R\cup B$
be a minimum DS-dominating set  with $|S|\le 3$. Without loss of generality, let $|S|=3$. We consider the following cases.

\textbf{Case 5.}
Suppose all three vertices of $S$ belong to $X\cup X'$.\\
Let $v_i'\notin S$, then we observe that
$N(v_i')=Y$.
As $S\cap Y=\varnothing$, the vertex $v_i'$ has no neighbour in $S$.
Contradicting the assumption that $S$ is a DS-dominating set.

Hence, all three vertices of $S$ cannot belong to $X\cup X'$.

\textbf{Case 6.}
Suppose all three vertices of $S$ belong to $Y\cup Y'$.\\
By an argument similar to that of  Case~5, we note that all three vertices of
$S$ cannot belong to $Y\cup Y'$.

\textbf{Case 7.}
Suppose $S$ contains vertices from both $X\cup X'$ and $Y\cup Y'$.\\
Without loss of generality, assume $S$ contains exactly one vertex from
$X\cup X'$. Then choose a twin vertex $u_j'\notin S$.
Since $N(u_j')=X$,
we note that $u_j'$ has only one neighbour in $X\cup X'$, which
contradicts the definition of a dual-server domination.

Hence no dual-server dominating set of cardinality at most 3 exists.
Therefore, $\gamma_{ds}(S_p(K_{m,n}))\ge 4.$\\
Combining this with the upper bound, we obtain \[
\gamma_{ds}(S_p(K_{m,n}))=  4.
\]
\end{proof}
\begin{observation}
\leavevmode
\begin{enumerate}
    \item{For the complete bipartite graph $K_{1,1}, \gamma_{ds}(Sp(K_{1,1}))= 3.$}
    \item{For the complete bipartite graph $K_{1,n}, \gamma_{ds}(Sp(K_{1,n}))= n+2.$}\end{enumerate}
\end{observation}
\begin{theorem}\label{SKN}
For the complete graph $K_n$, $
\gamma_{ds}(Sp(K_n)) = 4,  \text{ for all } n \ge 3.$
\end{theorem}

\begin{proof}
Let $K_n$ be the complete graph with vertex set $V(K_n)=\{v_1,v_2,\ldots,v_n\}$,
and let $V'=\{v_1',v_2',\ldots,v_n'\}$
be the set of twin vertices corresponding to $v_1,v_2,\ldots,v_n$ in the splitting graph $Sp(K_n)$.\\
We first show that $
\gamma_{ds}(Sp(K_n))\le 4.$
Consider the set $
S=R \cup B=\{v_1,v_2,v_1',v_2'\},$
where $R=\{v_1,v_1'\} \text{ and } B=\{v_2,v_2'\}.$\\
Suppose $v_i\notin S$. Since $K_n$ is complete, $v_i$ is adjacent to every original vertex. Also, $v_i$ is adjacent to every twin vertex except $v_i'$. In particular, $v_i$ is adjacent to $v_1,v_2,v_1'$ and $v_2'$. Hence, $v_i$ has at least one neighbour in both $R$ and $B$.
Now suppose $v_i'\notin S$. Since $v_i'$ has the same neighbourhood as $v_i$, it is adjacent to every original vertex except $v_i$. In particular, $v_i'$ is adjacent to $v_1$ and $v_2$. Hence, $v_i'$ has at least one neighbour in both $R$ and $B$.
Therefore, every vertex outside $S$ has at least one neighbour in both $R$ and $B$. Thus, $S$ is a DS-dominating set of $Sp(K_n)$. Consequently, \[
\gamma_{ds}(Sp(K_n))\le 4.\]

Next, we show that $
\gamma_{ds}(Sp(K_n))\ge 4.$\\
Suppose, to the contrary, that $\gamma_{ds}(Sp(K_n))\le 3.$
Let $S=R\cup B$
be a minimum DS-dominating set  of $Sp(K_n)$ with $|S|=3$.
Since $|S|=3$, either $R$ or $B$ contains exactly one vertex. Without loss of generality, assume that $R$ contains exactly one vertex.
Suppose $R=\{v_i\}$
for some $i$. Since $v_i'$ is not adjacent to $v_i$, the vertex $v_i'$ has no neighbour in $R$, contradicting the fact that $S$ is a DS-dominating set.
Similarly, suppose $R=\{v_i'\}$
for some $i$. Since $v_i$ is not adjacent to $v_i'$, the vertex $v_i$ has no neighbour in $R$, again contradicting the fact that $S$ is a DS-dominating set.

Hence, no DS-dominating set of $Sp(K_n)$ has cardinality 3. Therefore,
\[
\gamma_{ds}(Sp(K_n))\ge 4.
\]
\end{proof}
\begin{corollary}
For the complete graph $K_n$, $\gamma_{ds}(Sp(K_n))= 2\gamma_{ds}(K_n).$
\end{corollary}
\begin{proof}
It is known that $\gamma_{ds}(K_n)= 2$ (Observertion 1 of \cite{DS}). Also, by theorem~\ref{SKN}, $\gamma_{ds}(Sp(K_n))= 4$. Therefore, $\gamma_{ds}(Sp(K_n))= 4= 2\gamma_{ds}(K_n).$
\end{proof}
\section{Dual-Server Domination under Graph Operations}

In this section, we study the behaviour of the dual-server domination number under several graph operations. In particular, we consider the join, corona, and Cartesian product operations and establish corresponding results for these 
graph classes.
\subsection{Join Graphs}
\begin{definition}
[\textbf{Join of Graphs }\cite{West}]
The join of two graphs $G_1(V_1,E_1)$ and $G_2(V_2,E_2)$, denoted by $G_1+G_2$, is the graph obtained from the disjoint union $G_1\cup G_2$ by adding all edges joining every vertex of $V_1$ to every vertex of $V_2$.
\end{definition}
\begin{theorem}\label{DJP2}
Let $G$ and $H$ be graphs of orders $m \geq 2$ and $n \geq 2$, respectively. Then $\gamma_{ds}(G+H)= 2$ if and only if at least one of the following conditions holds:
\begin{enumerate}
\item {$\gamma_{ds}(G)= 2$} or {$\gamma_{ds}(H)= 2$} 
\item{$\Delta(G)= m-1 \text{ and } \Delta(H)= n-1$}
\end{enumerate}
\end{theorem}
\begin{proof}
Assume that $\gamma_{ds}(G+H)=2$.
Let $S= R \cup B= \{x,y\}$
be a minimum dual-server dominating set of $G+H$, where $
R=\{x\}$ and $B=\{y\}$.

Suppose, to the contrary,  that neither of the stated conditions holds. Then: 
\begin{enumerate}
\item{$\gamma_{ds}(G)\neq 2$} and {$\gamma_{ds}(H)\neq 2$}
\item{$\Delta(G)\le m-2$ or 
$\Delta(H)\le n-2$
}
\end{enumerate}
We now consider the following cases.

\textbf{Case 1.} Suppose $x,y\in V(G)$.\\
Since $G+H$ is the join of $G$ and $H$, every vertex of $H$ is adjacent to
both $x$ and $y$. Moreover, every vertex of $G\setminus\{x,y\}$ is adjacent
to both $x$ and $y$. Hence, $S$ is also a dual-server dominating set of $G$.
Therefore,
$\gamma_{ds}(G)=2$,
contradicting the assumption that $
\gamma_{ds}(G)\neq 2$.

\textbf{Case 2.} Suppose $x,y\in V(H)$.
By an argument similar to that in Case~1, we obtain $
\gamma_{ds}(H)=2$,
contradicting the assumption that $\gamma_{ds}(H)\neq 2$.

\textbf{Case 3.} Suppose, without loss of generality, that $x\in V(G)\quad\text{and}\quad y\in V(H)$.
Since $G+H$ is the join of $G$ and $H$, every vertex of
$H\setminus\{y\}$ is adjacent to $x$, and every vertex of
$G\setminus\{x\}$ is adjacent to $y$.
Since $S$ is a  DS-dominating set of $G+H$, every vertex of
$G\setminus\{x\}$ must also be adjacent to $x$, and every vertex of
$H\setminus\{y\}$ must also be adjacent to $y$.\\
Hence, $
\deg_G(x)=m-1 \quad\text{ and }\quad  \deg_H(y)=n-1 $.\\
Therefore, $
\Delta(G)=m-1
\quad\text{ and }\quad
\Delta(H)=n-1,$
which contradicts the assumption that $
\Delta(G)\le m-2
\quad\text{ or }\quad
\Delta(H)\le n-2.$

Hence, at least one of the conditions \textbf{(1)} or \textbf{(2)} holds.\\ 
Conversly, assume that atleast one of the given conditions holds. We consider the following cases.

\textbf{Case 4.} $\gamma_{ds}(G)= 2$.\\ Let $S= R \cup B= \{x, y\}$, without loss of generality where $R= \{x\}$ and $B= \{y\}$ be a minimum  DS-dominating set of $G$. Every vertex in $V(G)\setminus S$ is adjacent to both $x$ and $y$. Moreover, by the definition of the join, every vertex of $H$ is adjacent to both $x$ and $y$ . Therefore S is a minimum  DS-dominating set of $G+H$ also. Thus $\gamma_{ds}(G+H)= 2$. 

\textbf{Case 5.} Similarly if $\gamma_{ds}(H)= 2$, then $\gamma_{ds}(G+H)= 2$. 

\textbf{Case 6.} $\Delta(G)=m-1$  and  $\Delta(H)=n-1$.\\ Let $S=R \cup B$, where, without loss of generality $R= \{x\}$ and $B= \{y\}$. Let $x \in V(G)$ and $y \in V(H)$ be vertices such that $\deg_G(x)=m-1$  and   $\deg_H(y)=n-1$ respectively.\\ Since $\deg_G(x)=m-1$ , the vertex $x$ is adjacent to every vertex in $V(G) \setminus \{x\}$. Also, by the definition of the join, $x$ is adjacent to every vertex of $H$. Similarly, $\deg_H(y)=n-1$ implies that $y$ is adjacent to every vertex in  $V(H) \setminus \{y\}$, and by the definition of the join, $y$ is adjacent to every vertex of $G$.\\ Hence, every vertex in $V(G+H)\setminus S$ is adjacent to $x$ and $y$. Therefore, S is a DS-dominating set of $G+H$, and so $\gamma_{ds}(G+H) \leq 2$. Since every dual-server dominating set has cardinality at least 2, we conclude that $\gamma_{ds}(G+H)= 2$.
\end{proof}

\begin{theorem}\label{DJP3}
Let $G$ and $H$ be graphs of orders $m \geq 2$ and $n \geq 2$, respectively. Suppose that $\gamma_{ds}(G+H) \neq  2$. Then $\gamma_{ds}(G+H)= 3$ if and only if at least one of the following conditions holds:
\begin{enumerate}
\item{$\gamma_{ds}(G)= 3$} or {$\gamma_{ds}(H)= 3$} 
\item{$\Delta(G) \geq m-2$} or {$\Delta(H) \geq n-2$}
\end{enumerate}
\end{theorem}
\begin{proof}
Assume that $
\gamma_{ds}(G+H)=3.$
Let $S=R\cup B=\{x,y,z\}$
be a minimum dual-server dominating set of $G+H$, where, without loss of generality $R=\{x,y\}$ and $ B=\{z\}$.

Suppose, to the contrary, that 
\begin{enumerate}
\item{$\gamma_{ds}(G)\neq 3$} and { $\gamma_{ds}(H)\neq 3$}
\item{$\Delta(G)\le m-3$} and {$\Delta(H)\le n-3$}
\end{enumerate}

We now consider the following cases.

\textbf{Case 1.} Suppose $x,y,z\in V(G)$. \\Since $G+H$ is the join of $G$ and $H$, every vertex of $H$ is adjacent to
each of $x$, $y$ and $z$. Moreover, every vertex in
$V(G)\setminus S$ is adjacent to at least one vertex of $R$ and  $B$. Hence, $S$ is also a minimum DS-dominating set of $G$. Therefore, $\gamma_{ds}(G)=3$, contradicting the assumption that
$\gamma_{ds}(G)\neq 3.$ 

\textbf{Case 2.} Suppose $x,y,z\in V(H)$.\\By an argument similar to that in Case~1, we obtain $\gamma_{ds}(H)=3$,
contradicting the assumption that $\gamma_{ds}(H)\neq 3.$ 

\textbf{Case 3.} Suppose two vertices of $S$ belong to $G$ (or $H$), and
the remaining vertex belongs to $H$ (or $G$).

\textbf{Subcase 3.1.} Suppose, without loss of generality, $x,y\in V(G)$ and $z\in V(H).$\\
Since we have assumed $x$ and $y$ have degree atmost $m-3$ then 
there exists a vertex $v\in V(G)\setminus\{x,y\}$
that is nonadjacent to both $x$ and $y$.
Although $v$ is adjacent to $z$ by the join operation, it has no neighbour
in
$R=\{x\}$,
contradicting the assumption that $S$ is a minimum DS-dominating set of $G+H$.\\
Hence, $\Delta(G)\ge m-2$, 
contradicting the assumption that $
\Delta(G)\le m-3$ .

\textbf{Subcase 3.2.} Suppose, without loss of generality, $z\in V(G)$ and $x,y\in V(H).$\\  By an argument similar to that in Subcase~3.1, we obtain $
\Delta(H)\ge n-2$,
contradicting the assumption that $
\Delta(H)\le n-3$.

Hence, at least one of the conditions \textbf{(1)} or \textbf{(2)} holds. \\
Conversely, assume that at least one of the given conditions holds. We consider the following cases.

\textbf{Case 4.} Suppose that $\gamma_{ds}(G)=3$\\. Let $S=\{x,y,z\}$
be a minimum dual-server dominating set of $G$, where, without loss of generality, $R=\{x\} \quad \text{and} \quad B=\{y,z\}$. Then every vertex in $V(G)\setminus S$ is adjacent to at least one vertex in $R$ and $B$. Since $G+H$ is the join of $G$ and $H$, every vertex of $H$ is adjacent to every vertex of $S$. Therefore, $S$ is also a minimum dual-server dominating set of $G+H$. Hence,
\[
\gamma_{ds}(G+H)= 3.
\]

\textbf{Case 5.} Suppose that $\gamma_{ds}(H)=3$.\\ By an argument analogous to that in Case~4, we conclude that
\[
\gamma_{ds}(G+H)=3.
\]

\textbf{Case 6.} Suppose that $\Delta(G)\ge m-2$  or  $\Delta(H)\ge n-2$.
We consider the following subcases.\\
Let $S=\{x,y,z\}$,
where, without loss of generality, $R=\{x\}$ and  $B=\{y,z\}$.

\textbf{Subcase 6.1.} Let $x\in V(G)$ be a vertex with $\deg_G(x)=m-1$.
Choose any other vertex $y\in V(G)$, and let $z\in V(H)$.\\
Since $G+H$ is the join of $G$ and $H$, every vertex in $H\setminus\{z\}$ is adjacent to both $x$ and $y$. Moreover, every vertex in $V(G)\setminus\{y\}$ is adjacent to $x$ in $G$ and to $z$ in $H$ by the join operation. Hence, every vertex outside $S$ has a neighbour in both $R$ and $B$. Therefore, $S$ is a DS-dominating set of $G+H$. Thus,
\[
\gamma_{ds}(G+H)\le 3.
\]
Since $\gamma_{ds}(G+H)\neq 2$, we obtain
\[
\gamma_{ds}(G+H)=3.
\]

\textbf{Subcase 6.2.} Let $x\in V(H)$ be a vertex with $\deg_H(x)=n-1$.
Choose any other vertex $y\in V(H)$, and let $z\in V(G)$.
Applying the same argument as in Subcase~6.1, we conclude that
\[
\gamma_{ds}(G+H)=3.
\]

\textbf{Subcase 6.3.} Let $x\in V(G)$ be a vertex with $\deg_G(x)=m-2.$
Let $y$ be the unique vertex of $G$ that is not adjacent to $x$, and choose $z\in V(H)$.\\
Every vertex in $H\setminus\{z\}$ is adjacent to both $x$ and $y$ by the join operation. Furthermore, every vertex in $V(G)\setminus\{y\}$ is adjacent to $x$ in $G$ and to $z$ in $H$. Hence, every vertex outside $S$ is adjacent to at least one vertex in $R$ and  $B$. Therefore, $S$ is a DS-dominating set of $G+H$, and
\[
\gamma_{ds}(G+H)\le 3.
\]
Since $\gamma_{ds}(G+H)\neq 2$, we obtain
\[
\gamma_{ds}(G+H)=3.
\]

\textbf{Subcase 6.4.} Let $x\in V(H)$ be a vertex with $\deg_H(x)=n-2.$
Let $y$ be the unique vertex of $H$ that is not adjacent to $x$, and choose $z\in V(G)$.\\
By the same argument as in Subcase~6.3, we conclude that
\[
\gamma_{ds}(G+H)=3.
\]
\end{proof}

\begin{theorem}
Let $G$ and $H$ be graphs of orders $m\ge2$ and $n\ge2$, respectively. Then $\gamma_{ds}(G+H)=4$
if and  only if  following conditions holds:
\begin{enumerate}
    \item $\gamma_{ds}(G)\ge4$ and $\gamma_{ds}(H)\ge4$
    \item $\Delta(G)\le m-3$ and $\Delta(H)\le n-3$.
\end{enumerate}
\end{theorem}

\begin{proof}
Assume that $\gamma_{ds}(G+H)=4$.

We first show that condition (1) holds. Suppose, to the contrary, that $
\gamma_{ds}(G)\le3 \quad \text{or} \quad \gamma_{ds}(H)\le3.$

Suppose $\gamma_{ds}(G)\le3$. Let $S=R\cup B$ be a DS-dominating set  of $G$ with
$|S|\le3$. Then every vertex in $V(G)\setminus S$ is adjacent to at least
one vertex in $R$ and  $B$. Since $G+H$ is the join of
$G$ and $H$, every vertex of $H$ is adjacent to  $S$.
Therefore, $S$ is also a DS-dominating set of $G+H$, which implies that
\[
\gamma_{ds}(G+H)\le3,
\]
contradicting the assumption that $\gamma_{ds}(G+H)=4$. Hence,
\[
\gamma_{ds}(G)\ge4.
\]
Similarly, if $\gamma_{ds}(H)\le3$, then we can show by the similar argument,
\[
\gamma_{ds}(H)\ge4.
\]
Therefore, condition (1) holds.

Next, we show that condition (2) holds. Suppose, to the contrary, that
\[
\Delta(G)\ge m-2 \quad \text{or} \quad \Delta(H)\ge n-2.
\]
Then, by Theorem~\ref{DJP3},
\[
\gamma_{ds}(G+H)=3,
\]
which contradicts the assumption that $\gamma_{ds}(G+H)=4$. Therefore,
\[
\Delta(G)\le m-3
\quad \text{and} \quad
\Delta(H)\le n-3.
\]

Conversely, assume that the given conditions hold. Let $S=R\cup B=\{u,x,y,z\}$, 
where, without loss of generality, $R=\{u,x\}
\quad \text{and} \quad
B=\{y,z\}$.

Assume that $u,y\in V(G)
\quad \text{and} \quad x,z\in V(H)$.
Since $G+H$ is the join of $G$ and $H$, every vertex in
$V(G)\setminus\{u,y\}$ is adjacent to both $x$ and $z$, while every vertex in
$V(H)\setminus\{x,z\}$ is adjacent to both $u$ and $y$. Hence every vertex outside $S$ has a neighbour in both $R$ and $B$. Therefore, $S$ is a DS-dominating set of
$G+H$. Thus,
\[
\gamma_{ds}(G+H)\le4.
\]

We now show that $
\gamma_{ds}(G+H)\ge4$.

Suppose, to the contrary, that $
\gamma_{ds}(G+H)\le3$.
Let $S$ be a minimum DS-dominating set  of $G+H$ with $|S|\le3$. We consider the following
cases.

\textbf{Case 1.} Suppose $S\subseteq V(G)$.\\
Since $G+H$ is the join of $G$ and $H$, every vertex of $H$ is adjacent to
every vertex of $S$. Moreover, every vertex in $V(G)\setminus S$ is adjacent
to at least one vertex in $R$ and  $B$. Hence $S$ is
also a DS-dominating set of $G$. Therefore,
$
\gamma_{ds}(G)\le3$,
contradicting the assumption that
\[
\gamma_{ds}(G)\ge4.
\]

\textbf{Case 2.} Suppose $S\subseteq V(H)$.\\
By an argument similar to that in Case~1,  we show that $S$ is also a DS-dominating set of $H$.
Therefore, $
\gamma_{ds}(H)\le3$,
contradicting the assumption that
\[
\gamma_{ds}(H)\ge4.
\]

\textbf{Case 3.} Suppose that $S$ contains vertices from both $G$ and $H$.
We consider the following subcases.

\textbf{Subcase 3.1.} Suppose $\gamma_{ds}(G+H)=2$.

Then, by theorem~\ref{DJP2},
\[
\Delta(G)=m-1
\quad \text{ and } \quad
\Delta(H)=n-1,
\]
which contradicts the assumption that
\[
\Delta(G)\le m-3
\quad \text{and} \quad
\Delta(H)\le n-3.
\]

\textbf{Subcase 3.2.} Suppose $\gamma_{ds}(G+H)=3$.

Then, by  theorem~\ref{DJP3},
\[
\Delta(G)\ge m-2
\quad \text{or} \quad
\Delta(H)\ge n-2,
\]
which contradicts the assumption that
\[
\Delta(G)\le m-3
\quad \text{and} \quad
\Delta(H)\le n-3.
\]

Therefore,
\[
\gamma_{ds}(G+H)\ge4.
\]
\end{proof}

\begin{observation}
\leavevmode
\begin{enumerate}
\item{Let $G$ and $H$ be two graphs, each of order 1. Then $\gamma_{ds}(G+H)= 2$.}
\item{Let $G$ and $H$ be two graphs and let w be the order of H. Then $\gamma_{ds}(G+H)= w$ if and only if $G \cong K_1$, and $H \cong \overline{K_n},  K_2,  K_2 \cup \overline{K_{n-2}}$.
\item{If $G \cong K_{1}$ and $H \cong nK_{2}$, then $\gamma_{ds}(G+H)=n+1$.}}
\end{enumerate}
\end{observation}
\subsection{Cartesian Product of Graphs}
\begin{definition}[\textbf{Cartesian Product of Graphs} {\cite{West}}]
The Cartesian product of two graphs $G_1(V_1,E_1)$ and $G_2(V_2,E_2)$, denoted by
$G_1 \times G_2$ is the graph with vertex set $V_1 \times V_2$.
If two vertices $(u_1,u_2)$ and $(v_1,v_2)$ are adjacent in $G_1 \times G_2$
whenever either $u_1 = v_1$ and $u_2$ is adjacent to $v_2$ in $G_2$, or
$u_2 = v_2$ and $u_1$ is adjacent to $v_1$ in $G_1$.
\end{definition}

\begin{theorem}
Let $G$ and $H$ be graphs of orders $m\ge1$ and $n\ge1$, respectively. Then the dual-server domination number of the cartesian product  $G\times H $ is 
$\gamma_{ds}(G\times H)=mn$
if and only if $
\Delta(G)+\Delta(H)\le1$.
\end{theorem}

\begin{proof}
Let $
V(G)=\{u_1,u_2,\ldots,u_m\} \quad \text{ and }\quad
V(H)=\{v_1,v_2,\ldots,v_n\}$.
Then $
|V(G\times H)|=mn$.

Assume that $\gamma_{ds}(G\times H)=mn$. Let $S=R\cup B $
be a minimum dual-server dominating set of $G\times H$. Then $|S|=mn$. Now, suppose that  $\Delta(G)+\Delta(H)\ge2$.
We consider the following cases.

\textbf{Case 1.} $\Delta(G)\ge2$.\\
In particular, without loss of generality, let  $u_q$ and $u_r$ be two vertices adjacent to $u_k$. Choose any vertex $v_\ell\in V(H)$.

By the definition of the Cartesian product,
\[
(u_q,v_\ell)\sim(u_k,v_\ell)
\quad\text{and}\quad
(u_r,v_\ell)\sim(u_k,v_\ell).
\]
Hence the vertex $(u_k,v_\ell)$ has two distinct neighbours in
$G\times H$.
Without loss of generality, take
$
(u_q,v_\ell)\in R
\quad\text{and}\quad
(u_r,v_\ell)\in B.
$ Then $S \setminus (u_k,v_\ell) $  is not a DS-dominating set
, which implies
\[
|S| \neq mn,
\]
contradicting the assumption that $
\gamma_{ds}(G\times H)=mn $.

Therefore,
\[
\Delta(G)\le1.
\]

\textbf{Case 2.} $\Delta(H)\ge2$.\\
By an argument similar to that in Case~1, we obtain
\[
|S|\neq mn,
\]
contradicting the assumption that $
\gamma_{ds}(G\times H)=mn $.

Hence,
\[
\Delta(H)\le1.
\]

\textbf{Case 3.} $\Delta(G)=1$ and $\Delta(H)=1$.\\
Since $\Delta(G)=1$ and $\Delta(H)=1$, every vertex of $G$ and $H$ has at
most one neighbour.
Let $u_k$ be adjacent to $u_q$ in $G$, and let $v_\ell$ be adjacent to
$v_r$ in $H$. Then the four vertices $
(u_k,v_\ell),\;
(u_k,v_r),\;
(u_q,v_\ell),\;
(u_q,v_r)$
form a cycle of length $4$ in $G\times H$.\\
In particular, $(u_k,v_\ell)$ and $(u_q,v_r)$  adjacent to both
$(u_q,v_\ell)$ and $(u_k,v_r)$. Therefore, the vertex $(u_k,v_\ell)$  and $(u_q,v_r) $ have
two distinct neighbours in $G\times H$.\\
without loss of generality, take
\[
(u_q,v_\ell)\in R
\quad\text{and}\quad
(u_k,v_r)\in B.
\]
Hence, $
S\setminus\{(u_k,v_\ell),(u_q,v_r)\}
$ is not a dual-server dominating set. Hence 
\[
|S| \neq mn,
\]
which contradicts the assumption that
\[
\gamma_{ds}(G\times H)=mn.
\]

Thus the case $\Delta(G)=1$ and $\Delta(H)=1$ is impossible.

Therefore,
\[
\Delta(G)+\Delta(H)\le1.
\]

Conversely, suppose that 
$\Delta(G)+\Delta(H)\le1$.
Then one of the following cases must occur.

\textbf{Case 4.}  Suppose $\Delta(G)=0$ and $\Delta(H)=0$.\\
Then $V(G\times H)=\overline{K_{mn}} $. Therefore, \[
\gamma_{ds}(G\times H)=|V(G\times H)|=mn.
\]

\textbf{Case 5.} $\Delta(G)=1$ and $\Delta(H)=0$.\\
Since $\Delta(G)=1$, every vertex of $G$ has at most one neighbour.
Consequently, every vertex of $G\times H$ also has at most one neighbour.

Let $S=R\cup B$
be a DS-dominating set of $G\times H$. Every vertex outside $S$ must be adjacent to at
least one vertex in $R$ and $B$.

However, every vertex of $G\times H$ has degree at most one. Hence no
vertex can have two distinct neighbours. Therefore, no vertex can lie
outside $S$, and so $
S=V(G\times H)$.

Hence,
\[
\gamma_{ds}(G\times H)=|V(G\times H)|=mn.
\]

\textbf{Case 3.} $\Delta(G)=0$ and $\Delta(H)=1$.

By an argument similar to that in Case~2, we obtain
\[
\gamma_{ds}(G\times H)=mn.
\]
\end{proof}
\begin{observation} 
\leavevmode
\begin{enumerate}
\item{Let $G$ and $H$ be two graphs, each of order 1. Then $\gamma_{ds}(G\times H)= 1$.}
\item{Let  $G \cong K_1$, and $H \cong {K_n}$, where $n \geq 2$. Then $\gamma_{ds}(G\times H)= 2$}
\end{enumerate}
\end{observation}
\subsection{Corona of Graphs}
\begin{definition}[\textbf{Corona of Graphs} {\cite{Corona}}]
The corona of two graphs $G_1$ and $G_2$ is the graph $G = G_1 \circ G_2$
formed from one copy of $G_1$ and $|V(G_1)|$ copies of $G_2$, where the
$i$-th vertex of $G_1$ is adjacent to every vertex in the $i$-th copy of $G_2$.
\end{definition}
\begin{theorem}
Let $G$ and $H$ be any two graphs with $|V(G)|=m$. Then the dual-server domination number of the corona graph $G\circ H$ is
\[
\gamma_{ds}(G\circ H)=
\begin{cases}
m \cdot \min\{\gamma_{ds}(H),\,1+\gamma(H)\}, & \text{if } H\not\cong K_{1},\\[2mm]
m+\gamma(G), & \text{if } H\cong K_{1}.
\end{cases}
\]
\end{theorem}

\begin{proof}
Suppose $H\not\cong K_{1}$. Let $H^{v}$ denote the copy of $H$ corresponding to each vertex $v\in V(G)$. We first prove the upper bound by considering two cases.

\textbf{Case 1.} Suppose $
\gamma_{ds}(H)\le 1+\gamma(H)$.\\
For each $v\in V(G)$, let $S^{v}=R^{v}\cup B^{v}$ be a minimum DS-dominating set of $H^{v}$. Define $
S=R\cup B=\bigcup_{v\in V(G)}S^{v}$,
where $
R=\bigcup_{v\in V(G)}R^{v}$ and 
$B=\bigcup_{v\in V(G)}B^{v}$.\\
Since each copy contributes $\gamma_{ds}(H)$ vertices, $
|S|=m \cdot \gamma_{ds}(H)$.

Let $x\notin S$. Since $S^{v}$ is a DS-dominating set of $H^{v}$, the vertex $x$ has at least one neighbour in $R^{v}$ and at least one neighbour in $B^{v}$ for some $v \in V(G)$. Hence $x$ has at least one neighbour in both $R$ and $B$.

Let $x=v\notin S$. Since $v$ is adjacent to every vertex of $H^{v}$, and  $S^{v}$ contains vertices from both $R^{v}$ and $B^{v}$, the vertex $v$ has at least one neighbour in both $R$ and $B$.

Therefore, every vertex outside $S$ has at least one neighbour in both $R$ and $B$. Hence $S$ is a dual-server dominating set of $G\circ H$. Consequently,
\[
\gamma_{ds}(G\circ H)\le m \cdot \gamma_{ds}(H).
\]

\textbf{Case 2.} Suppose
$
\gamma_{ds}(H)\ge 1+\gamma(H)$.\\
For each $v\in V(G)$, let $D^{v}$ be a minimum dominating set of $H^{v}$. Let $S=R\cup B$,
where $
R=\bigcup_{v\in V(G)}\{v\}$ and 
$B=\bigcup_{v\in V(G)}D^{v}$.\\
Since each copy contributes one vertex from $G$ and $\gamma(H)$ vertices from $H^v$, $
|S|=m \cdot (1+\gamma(H))$.

Let $x\notin S$. Since $D^{v}$ is a dominating set of $H^{v}$, the vertex $x$ has a neighbour in $B$. Moreover, $x$ is adjacent to the vertex $v\in R$. Thus $x$ has at least one neighbour in both $R$ and $B$.

Hence every vertex outside $S$ has at least one neighbour in both $R$ and $B$. Therefore $S$ is a DS-dominating set of $G\circ H$, and
\[
\gamma_{ds}(G\circ H)\le m \cdot (1+\gamma(H)).
\]
Combining both cases, we obtain
\[
\gamma_{ds}(G\circ H) \leq
m \cdot \min\{\gamma_{ds}(H),\,1+\gamma(H)\}.
\]

Now we prove the lower bound. Let $S=R \cup B$ be a minimum DS-dominating set of $G\circ H$. We consider the following cases.

\textbf{Case 3.} $v\notin S$.

  Since $v\notin S$, no vertex of $H^v\setminus S$ can rely on $v$ to satisfy dual-server domination condition. Thus, both required neighbours of every vertex in $H^{v}\setminus S$ must lie in $H^v$. Therefore, $
S\cap V(H^{v})$
is a DS-dominating set of $H^{v}$. Since $H^{v}$ is a copy of $H$,
\[
|S\cap V(H^{v})|\ge\gamma_{ds}(H).
\]

\textbf{Case 4.}  $v\in S$.\\
Without loss of generality, assume that $v\in R$. Let $x\in V(H^v)$ and  $x\notin S$. Since $x$ is adjacent to $v$, it already has a neighbour in $R$. To satisfy the DS- domination condition, $x$ must also have a neighbour in $B$. Hence $
S\cap V(H^{v})$
contains a dominating set of $H^{v}$. Therefore,
\[
|S\cap V(H^{v})|\ge\gamma(H).
\]
Since $v\in S$, the subgraph $v+H^{v}$ contributes at least $
1+\gamma(H)$
vertices to $S$.

Hence, in either case, each subgraph $v+H^{v}$ contributes at least
\[
\min\{\gamma_{ds}(H),\,1+\gamma(H)\}
\]
vertices to $S$.

Since there are $m$ such subgraphs corresponding to the $m$ vertices of $G$, we obtain
\[
|S|\ge
m \cdot \min\{\gamma_{ds}(H),\,1+\gamma(H)\}.
\]

Therefore, $\gamma_{ds}(G\circ H)\ge
m \cdot \min\{\gamma_{ds}(H),\,1+\gamma(H)\}.$\\
Combining this with the upper bound, we obtain \[\gamma_{ds}(G\circ H)= 
m \cdot \min\{\gamma_{ds}(H),\,1+\gamma(H)\}.\]

Now suppose $H\cong K_{1}$.\\ For each $v\in V(G)$, let $u^{v}$ denote the unique  vertex  of the copy $H^v$.
Let $D$ be a minimum dominating set of $G$. Let
$ S=R\cup B$,
where $
R=\{u^{v}:v\in V(G)\}$ and $
B=D$.
Clearly, $
|S|=m+\gamma(G)$.

Then it directly follows from the definition that $S$ is a minimum dual-server dominating set. Therefore, 

\[
\gamma_{ds}(G\circ H)\ = m+\gamma(G).
\]
\end{proof}

\section*{Concluding Remarks}

In this paper, we investigated dual-server domination in several graph operations. We established exact values of the dual-server domination number for the join, Cartesian product, and corona of graphs, and determined the dual-server domination number of the splitting graphs of paths, cycles, complete graphs and complete bipartite graphs. These results contribute to the understanding of dual-server domination for graph operations.

The results obtained in this paper naturally lead to several directions for further investigation. It would be worthwhile to determine the dual-server domination number for other graph operations. Another natural direction is to investigate dual-server domination in the splitting graphs of other graph families.

\end{document}